\documentclass[a4paper,10pt]{article}
\usepackage[utf8]{inputenc}
\usepackage{amsmath,amssymb,mathrsfs,graphicx,amsthm,dsfont}
\usepackage{xfrac}

\date{Final version to appear in \\[2pt]
\emph{Chaos, Solitons \& Fractals} \\[3pt] 
Volume 111, June 2018, Pages 55--61
}

\title{Global-local mixing for the Boole map}
\author{
Claudio Bonanno\,\thanks{
Dipartimento di Matematica, Universit\`a di Pisa, Largo Bruno 
Pontecorvo 5, 56127 Pisa, Italy. E-mail: 
\texttt{claudio.bonanno@unipi.it}.}
,
Paolo Giulietti\,\thanks{
Scuola Normale Superiore - Centro di Ricerca Matematica Ennio 
De Giorgi, Piazza dei Cavalieri 7, 56126 Pisa, Italy.
E-mail: \texttt{paologiulietti.math@gmail.com}.}
,
Marco Lenci\,\thanks{
Dipartimento di Matematica, Universit\`a di Bologna,
Piazza di Porta San Donato 5, 40126 Bologna, Italy. 
E-mail: \texttt{marco.lenci@unibo.it}.}
\thanks{
Istituto Nazionale di Fisica Nucleare,
Sezione di Bologna, Via Irnerio 46,
40126 Bologna, Italy.}
}

\newtheorem{theorem}{Theorem}
\newtheorem{lemma}[theorem]{Lemma}
\newtheorem{proposition}[theorem]{Proposition}
\newtheorem{corollary}[theorem]{Corollary}
\newtheorem{definition}[theorem]{Definition}

\newtheorem{remark}[theorem]{Remark}

\newcommand{\fn} {function}
\newcommand{\me} {measure}

\newcommand{\erg} {ergodic}
\newcommand{\sy} {system}

\newcommand{\pr} {probability}
\newcommand{\dsy} {dynamical system}

\newcommand{\m} {mixing}
\newcommand{\ob} {observable}

\newcommand{\bR}{{\mathbb R}}
\newcommand{\bN}{{\mathbb N}}
\newcommand{\bP}{{\mathbb P}}
\newcommand{\bZ}{{\mathbb Z}}
\newcommand{\bC}{{\mathbb C}}

\newcommand{\ds} {\displaystyle}

\newcommand{\avg} {\overline{m}}

\begin{document}

\maketitle

\begin{abstract}
In the context of `infinite-volume \m' we prove global-local \m\ for the 
Boole map, a.k.a.\ Boole transformation, which is the prototype of a 
non-uniformly expanding map with two neutral fixed points. Global-local 
\m\ amounts to the decorrelation of all pairs of global and local observables.
In terms of the equilibrium properties of the system it means that 
the evolution of every absolutely continuous probability measure 
converges, in a certain precise sense, to an averaging functional over the 
entire space.

\end{abstract}

\section{Introduction}

In 1857 G.~Boole \cite{Boole1857} proved that 
\begin{equation} \label{eq:origboole}
  \int_{-\infty}^{\infty} f(x) \, dx = \int_{-\infty}^{\infty} f \! \left(x - \frac{1}{x}
  \right) dx ,
\end{equation}
whenever one of the two sides makes sense. He obtained this result by 
investigating changes of variables of the type $x - \frac{\lambda_1}{x - 
\lambda_2}$, for $\lambda_1, \lambda_2 \in \mathbb{R}$, which were 
extremely helpful in evaluating indefinite integrals, especially when 
complex contour integration was no trivial matter. In this note we 
look at some properties of the Boole map $T:\mathbb{R} \to \mathbb{R}$ 
where 
\[
T(x) := x - \frac{1}{x} .
\]
The graph of $T$ is shown in Fig.~\ref{fig-boole-map} below.

Eq.~(\ref{eq:origboole}) is equivalent to the property that, for every 
measurable $A \subseteq \bR$, $m(T^{-1} A) = m(A)$, where $m$
is the Lebesgue measure on $\bR$. In the language of \dsy s, we say 
that $m$ is invariant for $T$. This fact could also be ascertained via a 
simple direct computation for all $A = [a,b]$, which is no loss of 
generality.

Since $m$ is infinite, we are in the scope of infinite \erg\ theory. A 
number of techniques and results in this field have been given for
infinite-\me-preserving expanding maps of the unit interval with 
neutral fixed points, usually by means of suitable induction schemes; 
cf.~\cite{Aaronson97, ds, Zweimuller09}. It is not hard to represent $T$ 
in this fashion: for example, the conjugation $\psi:(0,1) \to \bR$ defined by 
\[
\psi(y) = \frac{1}{1-y} - \frac{1}{y} 
\]
gives rise to a two-branched expanding map $\bar{T} := 
\psi^{-1} \circ T \circ \psi :  (0,1) \to (0,1)$ which has neutral fixed 
points at $y=0$ and $y=1$. For this and further developments 
based on this approach see the lecture notes \cite{Zweimuller09}.

In this note we are interested in a mixing property of the Boole map which 
involves \emph{global \ob s} \cite{limix}: for this it is easier to view the 
system as a Lebesgue-\me-preserving map of $\bR$.

A global \ob, roughly speaking, is a function that is supported more or 
less evenly over the phase space, thus representing the observation of
an ``extensive'' quantity in space, as opposed to a \emph{local \ob}, 
which represents a quantity that is only relevant in a confined portion of 
the space. We show that global and local \ob s decorrelate in time, a 
property called \emph{global-local \m} \cite{limix}, see Definition 
\ref{def:GLM} below. This property provides interesting information on the 
stochastic properties of the \sy, as we will see.

Global-local \m\ has been proved for other systems as well, for example 
\dsy s representing random walks \cite{limix}, certain uniformly 
expanding Markov maps on $\bR$ \cite{lmmaps} and maps with one 
neutral fixed point \cite{BGL17}.  In \cite{BGL17} we claimed that the 
technique presented there was very flexible and could be adapted to 
a variety of different cases. Moreover, when presenting our results,
we were asked whether they held for the Boole map too, which is an 
important example and also the prototype of an expanding map with 
more than one neutral fixed point.

In what follows we give a complete proof of global-local \m\ for the 
Boole map and highlight some of its applications. This gives the 
interested reader a chance to see the techniques presented in 
\cite{BGL17} at work in a specific, relatively simple, case. 
Furthermore, we require a technical lemma that was stated with 
no proof in \cite[Remark 2.14]{BGL17}. We take the chance to 
present its proof --- in fact, a generalization thereof --- in the 
Appendix of this note.

\bigskip
\noindent
\textbf{Acknowledgments.}\ C.B.\ was partially supported by 
PRA2017, Universit\`a di Pisa, \emph{``Sistemi dinamici in analisi, 
geometria, logica e meccanica celeste''}. P.G.\ was partially 
supported by Instituto de Matem\'atica, Universidade Federal do 
Rio Grande do Sul, Porto Alegre, RS, Brazil. This research is part 
of the authors' activity within the DinAmicI community, see 
\texttt{www.dinamici.org}.

\section{Results}  

Let us recall that a \dsy\ $(M, \mathcal{B}, \mu, T)$, where 
$(M, \mathcal{B}, \mu)$ is a $\sigma$-finite \me\ space and $T$ is a
bi-measurable self-map of $M$, is called \emph{non-singular} if 
$\mu(A) = 0$ implies $\mu(T^{-1 }A) = 0$, for all $A \in \mathcal{B}$.  
Clearly, a \me-preserving \sy\ is non-singular.

An \erg\ property that is needed in our arguments later on is exactness:

\begin{definition} 
  The non-singular \dsy\ $(M, \mathcal{B}, \mu, T)$ is called \emph{exact} if
  the \emph{tail $\sigma$-algebra} $\bigcap_{n=0}^{\infty} T^{-n} \mathcal{B}$
  contains only null sets and complements of null sets, w.r.t.\ $\mu$.
\end{definition}

As is the case of many expanding maps with the right ``kneading''
properties, the Boole map $T$ is exact 
\cite[Exercise 1.3.4(6)]{Aaronson97}.

A well-established powerful tool to study the stochastic properties of
a map $T$ is the transfer operator $P$. If one chooses $\mu$ as the 
reference \me, $P$ is defined implicitly by the identity
\[
  \int_M  (F \circ T)\, g\, d\mu = \int_M F\, (P g)\, d\mu,
\]
where $F \in L^\infty(\mu)$ and $g \in L^1(\mu)$. One can use $P$ to 
give a useful criterion for exactness.

\begin{theorem}[Lin \cite{Lin71}] \label{thm-lin}
  A non-singular \dsy\ is exact if and only if for all $g \in L^1(\mu)$ such 
  that $\mu(g) = 0$ we have $\ds \lim_{n \to \infty} \| P^n g \|_1 = 0 $. 
\end{theorem}

\begin{corollary} \label{cor:lpmu}
  Suppose that $T$ is non-singular and exact and $\mu(M) = \infty$. Then, 
  for all $A, B$ with $\mu(A), \mu(B) < \infty$, $\ds \lim_{n \to \infty} 
  \mu(T^{-n}A \cap B) = 0$.
\end{corollary}

Corollary \ref{cor:lpmu} is not hard to show, using Lemma 
\ref{lem-main-pmu} below. In any case, it is a direct consequence
of \cite[Theorem 3.5(b)]{lpmu}. It is important because it shows that, at
least for exact systems, a naive transposition of the classical definition
of \m\ to the infinite-\me\ setting does not make sense. Actually, the 
conclusion of Corollary \ref{cor:lpmu} holds for a much larger class of \sy s, 
which A.~B.~Hajian and S.~Kakutani called \emph{of zero-type} \cite{hk}; 
cf.\ \cite{ds} for a ``modern'' version of the original definition of \cite{hk}.

Coming back to the Boole map $T: \bR \to \bR$ and the Lebesgue \me\
$m$, we now give the precise definitions of global and local \ob.

\begin{definition}
  A \emph{global \ob} is any $F \in L^\infty(m)$ for which the limit
  \begin{equation} \label{def-avg}
    \avg(F) := \lim_{a \to +\infty} \, \frac1{2a} \int_{-a}^a F \, dm 
  \end{equation}
  exists; $\avg(F) $ is called the \emph{infinite-volume average} 
  of $F$.
\end{definition}

\begin{definition}
  A \emph{local \ob} is any $g \in L^1(m)$. 
\end{definition}

For the sake of readability, global and local observables are indicated, 
respectively, with uppercase and lowercase letters. 

Good examples of global observables are the bounded periodic functions, 
say $F(x) = \sin(x)$ or
\begin{equation} \label{ex-per-fn}
  F(x) = \left\{
  \begin{array}{cl} 
    1 \,, & \mbox{ if } x \in [2n, 2n + 1) ,\ n \in \mathbb{Z} ; \\[2pt] 
    -1 \,,  & \mbox{ if } x \in [2n + 1, 2n+2) ,\ n \in \mathbb{Z} .  
  \end{array} 
  \right.
\end{equation}
In both cases $\avg(F)=0$. One can also consider \ob s that
distinguish between the two fixed points: for example, any bounded
$F: \bR \to \bC$ such that 
\begin{equation} \label{ex-two-lims}
  \lim_{x \to \pm \infty} F(x) = \ell_\pm,
\end{equation}
with $\ell_+ \ne \ell_-$. Evidently $\avg(F) = \frac12 (\ell_+ + \ell_-)$. Of 
course one can think of more complicated \ob s, for example \fn s 
that distinguish between the two fixed points, but do not have a 
limit at either of them, such as
\[
  F(x) = \frac{e^x}{e^x+1} + \cos\! \left( \frac{e^x+1}{e^x+2} \, x \right).
\]
In this particular case it can be checked that $\avg(F) = \frac12$.

Recalling the standard notation $m(f) := \int_\bR f \, dm$, the property 
we want to study for the Boole map is the following:

\begin{definition} \label{def:GLM}
  $T$ is said to be \emph{global-local \m} if, for all global \ob s $F$ 
  and local \ob s $g$,
  \[
    \lim_{n \to \infty} m( (F \circ T^n) g ) = \avg(F) m(g).
  \]
\end{definition}

There are stronger and weaker notions of mixing between global and 
local \ob s; cf.\ \cite{limix}: in recent literature \cite{lpmu, lmmaps, BGL17}
Definition \ref{def:GLM} has been indicated with the label {\bf (GLM2)}. It 
is perhaps the most natural way to reveal ``decorrelation'' between a 
global and a local \ob\ and has an interesting interpretation in terms of 
the dynamics of \me s under $T$. In fact, it can be proved that $T$ is global-local \m\ if and 
only if, for any absolutely continuous probability \me\ $\nu$ and any global 
\ob\ $F$,
\begin{equation} \label{eq:referee}
  \lim_{n \to \infty} \int_{\bR} F \, d (\nu \circ T^{-n}) = \avg(F).
\end{equation}
To see this it suffices to rewrite the above l.h.s.\ as $m((F \circ T^n)g)$, 
where $g := \frac{d\nu} {dm}$. Observe also that $m(g) = 1$. This shows 
that global-local mixing gives \eqref{eq:referee}.
The same argument proves that \eqref{eq:referee} implies the limit of 
Definition \ref{def:GLM} for all $g \ge 0$ with
$m(g)=1$. Since both sides of the limit are linear in $g$, the
conclusion is extended to all $g \ge 0$. For the general case one uses
the identity $g = g^+ - g^-$, where $g^+, g^-$ are, respectively, the
positive and negative parts of $g$. 

So the functional $\avg$, acting on the space of global \ob s, is the limit of 
the evolution of a large class of initial probability \me s.

As a simple application, consider the \ob\ $F$ of (\ref{ex-per-fn}). 
Proving that $T$ is global-local \m\ will show that for a random initial 
condition $x$, w.r.t.\ to any law $\nu \ll m$, the point $T^n(x)$ will be 
as likely to belong to an ``even'' as to an ``odd'' unit interval, for
$n \to \infty$. If we also consider the family of \ob s $F$ as in 
(\ref{ex-two-lims}), we see that the probability that $T^n(x)$ lies 
in a given neighborhood of $+\infty$, respectively $-\infty$, 
converges to $\frac12$ as $n \to \infty$.

Definition \ref{def:GLM} and its interpretation rely on the observation that $\avg$ 
is $T$-invariant on global observables. 

\begin{proposition}
  If $F$ is a global \ob, then so is $F \circ T$, with $\avg(F \circ T) =
  \avg(F)$.
\end{proposition}

\begin{proof}
Given $a>1$, it is easy to check that
\[
  T^{-1} \! \left[ -a+\frac1a \,,\, a-\frac1a \right] = \left[ -a \,,\, -\frac1a \right]
  \cup  \left[ \frac1a \,,\, a\right] .
\]
Using the $T$-invariance of $m$, we obtain
\[
  \int_{-a+\frac1a}^{a-\frac1a} F \, dm = \int_{-a}^{-\frac1a} F \circ T \, dm +
  \int_{\frac1a}^{a} F \circ T \, dm.
\]
Since $F$ is bounded, dividing both sides of the above identity by $2a$ 
and taking the limit $a \to +\infty$ gives the assertion.
\end{proof}

In the rest of this section we prove our main result:

\begin{theorem} \label{th:main}
  The Boole map is global-local mixing. 
\end{theorem}

\begin{proof}
An easy argument, based on Theorem \ref{thm-lin}, shows that it suffices 
to check the limit of Definition \ref{def:GLM} for one fixed local \ob. This 
was given in \cite[Lemma 3.6]{lpmu}; we restate it here for 
convenience.

\begin{lemma} \label{lem-main-pmu}
  Suppose that $T$ is exact and $F$ is a global \ob. If the equality
  \[
    \lim_{n \to \infty} m( (F \circ T^n) g) = \avg(F) m(g)  
  \]
  holds for one local \ob\ $g$ with $m(g) \ne 0$, then it holds for
  all local \ob s $g$.
\end{lemma}

We will find a family of local \ob s which satisfy the above limit
for all global \ob s. In order to do this, we utilize the Perron-Frobenius operator of $T$, that is, the transfer
operator $P$ w.r.t.\ $m$. Defining $\phi_+ := (T |_{\bR^+})^{-1}$ and
$\phi_- := (T |_{\bR^-})^{-1}$, $P$ is given by
\begin{equation} \label{eq:trop}
  (Pg)(x) = | \phi_+'(x) | \, g(\phi_+(x)) + | \phi_-'(x) | \, g(\phi_-(x)) .
\end{equation}
In explicit terms,
\[
  \phi_\pm (x) = \frac x2 \pm \xi(x) , \qquad
  \phi_\pm' (x) = \frac12 \pm \frac{x}{4\xi(x)},
\]
where $\xi(x) := \sqrt{\frac{x^2}{4} + 1}\,$, see Fig.~\ref{fig-boole-map}.
Evidently,
\begin{equation} \label{phi-symm}
  \phi_-(x) = - \phi_+(-x) , \qquad \phi'_-(x) = \phi'_+(-x).
\end{equation}

\begin{figure}[ht] 
\centerline{\includegraphics[width=8cm,keepaspectratio]{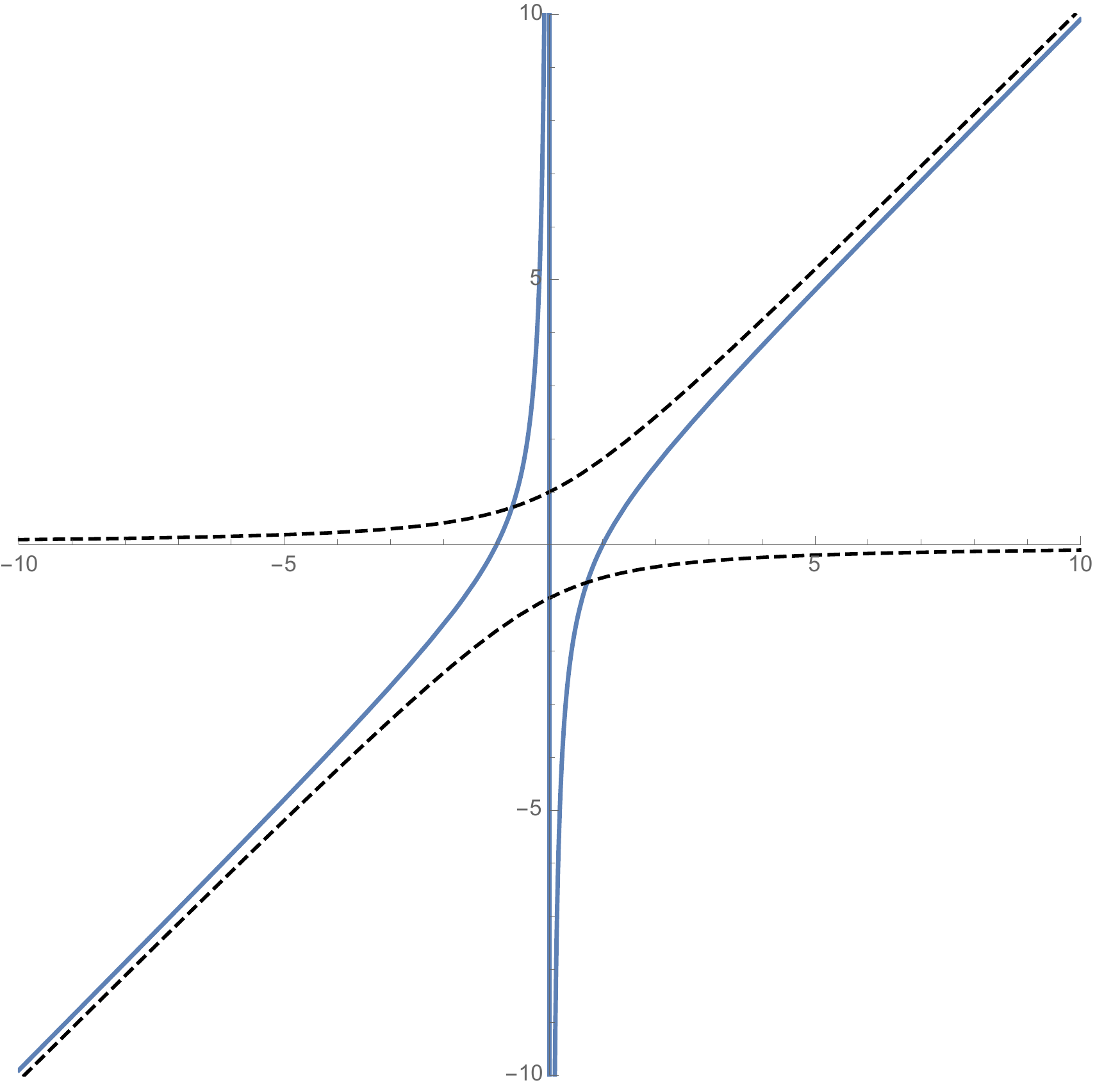}}
\caption{The Boole map $T$ (solid blue lines) and its two inverse branches
  $\phi_+, \phi_-$ (dashed black lines).}
\label{fig-boole-map}
\end{figure}


Let us now specialize to local observables $g$ that are even \fn s 
of $x$. In view of (\ref{phi-symm}), expression (\ref{eq:trop}) becomes
\begin{equation} \label{eq:trop-even}
  (Pg)(x) = \phi'_+(x) \, g(\phi_+(x)) + \phi'_+(-x) \, g(\phi_+(-x)) .
\end{equation}
This essentially coincides with the Perron-Frobenius operator of the 
map $\tilde T :\bR^+ \to \bR^+$ given by
\begin{equation} \label{tilde-t}
  \tilde T(x) := \left\{ 
  \begin{array}{ll}
     \ds x-\frac 1x \,, & \mbox { if } x\ge 1 ; \\[12pt] 
     \ds \frac 1x -x \,, & \mbox{ if } 0 < x \le 1 .
   \end{array} 
   \right.
\end{equation}
In fact, $\tilde T$ has two inverse branches, $\phi_0 : [0,+\infty) \to 
[1,+\infty)$ and $\phi_1 : [0,+\infty) \to (0,1]$, given by, respectively,
\begin{align*}
  \phi_0(x) &= \phi_+(x) = \frac x2 +\xi(x); \\
  \phi_1(x) &= \phi_+(-x) = -\frac x2 +\xi(x).
\end{align*}
The graphs of $\tilde T$, $\phi_0$ and $\phi_1$ are depicted in
Fig.~\ref{fig-new-map}.

\begin{figure}[ht] 
\centerline{\includegraphics[width=8cm,keepaspectratio]{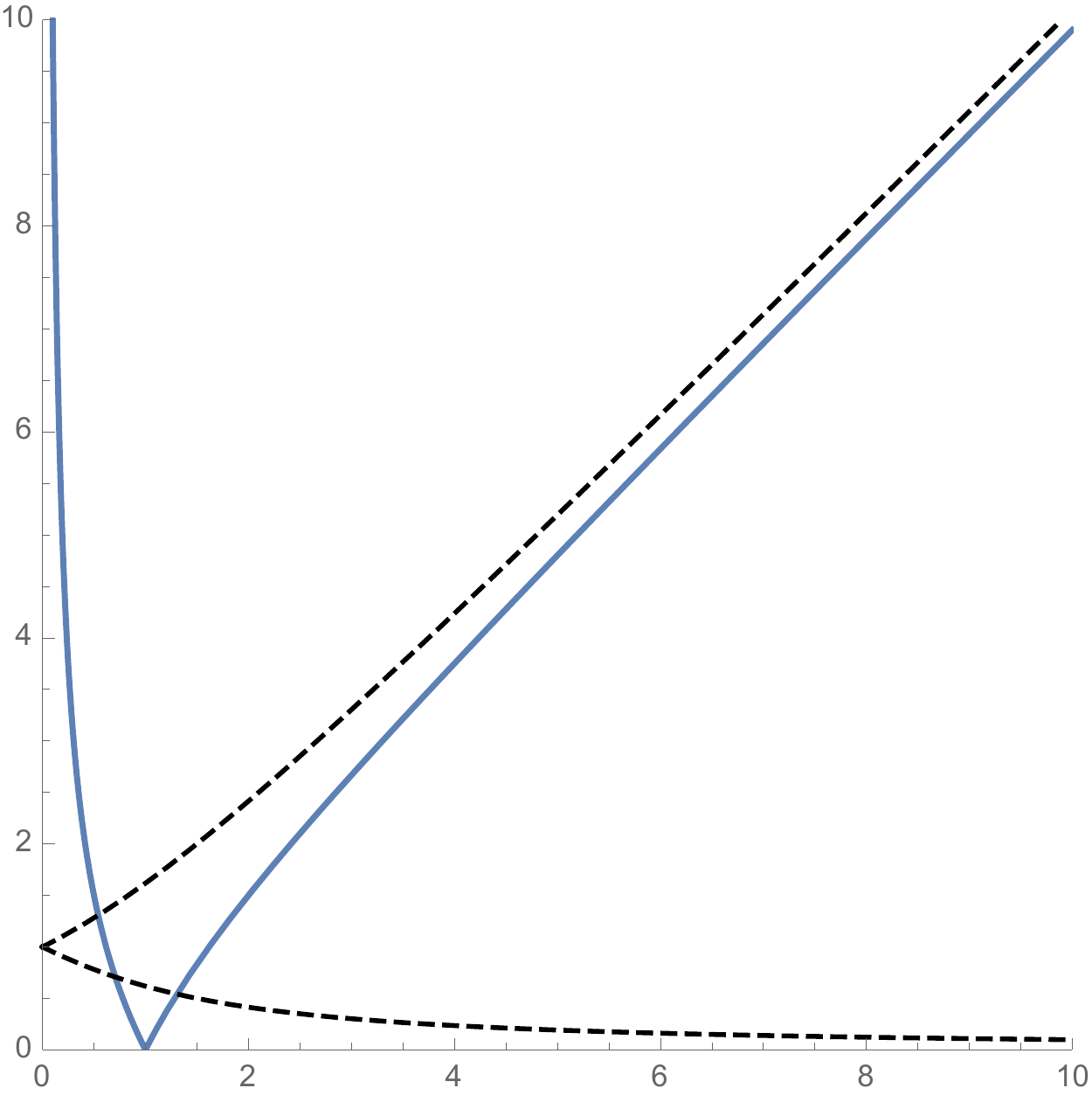}}
\caption{The map $\tilde T$ (solid blue lines) and its inverse branches
  $\phi_0, \phi_1$ (dashed black lines).}
\label{fig-new-map}
\end{figure}

%
Clearly, $\tilde T$ preserves the Lebesgue \me\ on $\bR^+$, which we
keep denoting by $m$. The Perron-Frobenius operator $\tilde P$ of $\tilde T$
acts on \fn s $\tilde g \in L^1(\bR^+, m)$ as
\[
  (\tilde P \tilde g)(x) = | \phi'_0(x) | \, \tilde g(\phi_0(x)) +
  | \phi'_1(x) | \, \tilde g(\phi_1(x)) .
\]
This expression coincides with (\ref{eq:trop-even})
restricted to $x > 0$.

We are thus reduced to studying  the transfer operator of the map $\tilde T$,
which belongs to the class of maps of the half-line studied in \cite{BGL17}. 
The problem is that in \cite{BGL17} global-local \m\ is proved for maps with 
all increasing branches, and the case of maps with one decreasing branch
is only briefly discussed in \cite[Remark 2.14]{BGL17}.  

The following technical lemma is crucial.

\begin{lemma} \label{th:lemma}
  Let $\tilde T$ be defined by (\ref{tilde-t}), with $\tilde P$ its 
  associated Perron-Frobenius operator. If
  \[
    {\cal F} := \left\{ \tilde g \in L^1(\bR^+,m) \,:\, \tilde g \in C^2,\, \tilde g>0,\, 
    \tilde g'<0,\, \tilde g'' + \tilde g'<0 \right\}
  \]
  then $\tilde P({\cal F})\subset {\cal F}$.
\end{lemma}

Lemma \ref{th:lemma} is proved in the Appendix, where we first show that
its assertion holds for a class of maps with one decreasing and 
one increasing branch, and then check that $\tilde T$ belongs to this class. 

The arguments outlined earlier imply the corresponding result for the Boole 
map $T$:

\begin{corollary} \label{th:lemma-boole}
  If $g\in L^1(\bR, m)$ and is $C^2$ in $\bR \setminus \{0\}$, even, positive and 
  such that $g'(x)<0$ and $g''(x) + g'(x) <0$ for all $x>0$, then $Pg$ has the 
  same properties.
\end{corollary}

We take a $g$ as in Corollary \ref{th:lemma-boole} to guarantee that,
for all $n$, $P^n g$ is even and strictly decreasing on $\bR^+$. In view of
Lemma \ref{lem-main-pmu}, Theorem \ref{th:main} will be proved once we 
show that for all global \ob s $F$,
\begin{equation} \label{goal}
  \lim_{n \to \infty} m(F (P^n g)) = \avg(F) m(g).
\end{equation}
We follow the proof of Lemma 4.3 of \cite{BGL17} except for a few twists. 
First of all, we assume $\avg(F) = 0$ and $m(g)=1$, otherwise one can prove 
(\ref{goal}) for $F_1 := F - \avg(F)$ and $g_1 := \frac{g}{m(g)}$, and then easily 
derive it for $F$ and $g$.

Fix $\varepsilon > 0$. By definition (\ref{def-avg}) there exists $\bar{a} > 0$ 
such that
\begin{equation} \label{main20}
  \forall a \ge \bar{a}, \quad \frac{1}{2a} \left| \int_{-a}^{a} F \, dm \right| 
  \le \frac{\varepsilon}{2}.
\end{equation}
For $x \in \bR$ and $n \in \bN$, set 
\[
  \gamma_n(x) = \gamma_{n,\bar{a}}(x) := \min \{ P^n g(\bar{a}), P^n g(x) \} .
\]
Thus, $\gamma_n $ is a positive, even \fn, flat on $[-\bar{a}, \bar{a}]$ and
strictly decreasing on $[\bar{a}, +\infty)$. It is a local observable because 
$\| \gamma_n \|_1 \le \| P^n g \|_1 =  \| g \|_1 = 1$. 

We defined $\gamma_n$ to use it in place of $P^n g$ in (\ref{goal}). This 
can be done to a better and better approximation as $n \to \infty$. The 
advantage is that,
thinking of $\gamma_n$ as a density w.r.t.\ which we integrate $F$, we can 
``slice'' this density into uncountably many horizontal segments, all longer 
than $2\bar{a}$. By virtue of (\ref{main20}), the integral of $F$ over each 
segment is very small, thus showing that the original integral is also
small. 

Let us make this idea precise. Write
\begin{equation} \label{main-goal2}
  m(F (P^n g)) = \int_{-\infty}^{+\infty} F \gamma_n \, dm +
  \int_{-\bar{a}}^{\bar{a}} F (P^n g - \gamma_n) \, dm =: \mathcal{I}_1 + 
  \mathcal{I}_2.
\end{equation}

We first consider $\mathcal{I}_2$. Since $P^n g \ge \gamma_n > 0$ we 
have
\[
  0 \le \int_{-\bar{a}}^{\bar{a}}  (P^n g - \gamma_n) \, dm \le 
  \int_{-\bar{a}}^{\bar{a}} P^n g \, dm = m\! \left( (1_{[-\bar{a}, \bar{a}]} \circ 
  T^n) g \right) ,
\]
where $1_{[-\bar{a}, \bar{a}]}$ is the indicator \fn\ of $[-\bar{a}, \bar{a}]$.
The fact that $T$ is of zero-type (see Corollary \ref{cor:lpmu} and following
paragraph) implies that the rightmost term above vanishes as $n \to \infty$.
Thus, for all sufficiently large $n$,
\begin{equation} \label{main50}
  | \mathcal{I}_2 |  \le
  \| F \|_\infty  \int_{-a}^{a} (P^n g - \gamma_n) \, dm \leq \frac{\varepsilon}{2}.
\end{equation}

Now for $\mathcal{I}_1$. For $0 < r < \max \gamma_n = \gamma_n (\bar{a})$,
let $a(r)$ be uniquely defined by
\[
  a(r) > \bar{a}, \quad \gamma_n (a(r)) = r.
\]
In other words, $a(r) = \left( \gamma_n |_{ (\bar{a}, +\infty) } \right)^{-1} (r)$.
We rewrite $\mathcal{I}_1$ as
\[ 
  \mathcal{I}_1 = \int_{-\infty}^{+\infty} \! F(x) \left( \int_0^{\gamma_n(x)} 
  \!\!\ dr \right) dx = \int_0^{\gamma_n (\bar{a})} \! \left( \int_{-a(r)}^{a(r)} \! 
  F(x) \, dx \right) dr, 
\]
where we have used Fubini's Theorem. Using (\ref{main20}) with 
$a := a(r)$, we estimate the above as follows:
\[
\begin{split}
  | \mathcal{I}_1 | &\le \int_0^{\gamma_n (\bar{a})} \left| \int_{-a(r)}^{a(r)} 
  \! F(x) \, dx \right| dr 
  \le \frac{\varepsilon}2 \int_0^{\gamma_n (\bar{a})} \! 2 a(r) \, dr \\
  &= \frac{\varepsilon}2 \, m(\gamma_n) \le \frac{\varepsilon}2 ,
  \end{split}
\]
because $m(\gamma_n) = \| \gamma_n \|_1 \le 1$. The above estimate holds 
uniformly in $n$. Together with (\ref{main-goal2}) and (\ref{main50}), it implies 
(\ref{goal}) for the case $\avg(F) = 0$ and $m(g)=1$. As already
discussed, this is enough to prove Theorem \ref{th:main}.
\end{proof}

\section{Application}

We conclude our exposition by presenting a more sophisticated application
of global-local \m\ than previously discussed (in the paragraphs after 
Definition \ref{def:GLM}). It has to do with the stochastic properties of our 
map, that is, with the interpretation of \ob s as random variables w.r.t.\ a 
random choice of the initial condition. The interested reader will find more 
details in Section 3.2 of \cite{BGL17}.

\begin{definition}
  Let $F_n$ be a sequence of measurable \fn s $\bR \to \bR$ and $X$ a 
  random variable on some probability space $(\Omega, \bP)$. We say that
  $F_n$ converges to $X$ in \emph{strong distributional sense}, as
  $n \to \infty$, if for all \pr\ \me s $\nu \ll m$ the distribution of $F_n$ 
  w.r.t.\ $\nu$ converges to that of $X$.
\end{definition}

\begin{proposition} \label{prop-new-app}
  Let $T$ be the Boole map and $F: \bR \to \bR$ a measurable \fn\ such
  that $\avg( e^{i \theta F} )$ exists for all $\theta \in \bR$. Then $F \circ T^n$ 
  converges in strong distributional sense, as $n \to \infty$, to the random 
  variable $X$ with characteristic \fn\ $\varphi_X(\theta) := \avg( e^{i \theta F} )$.
\end{proposition}

\begin{proof}
Given a probability \me\ $\nu \ll m$, let $g$ denote its Radon-Nikodym 
derivative $\frac{d\nu}{dm}$. For all $\theta \in \bR$, $e^{i \theta F}$ is a 
global \ob\ by hypothesis (observe that $F$ takes values in $\bR$). 
Thus, by global-local \m,
\[
  \lim_{n \to \infty} \nu ( e^{i \theta F \circ T^n} ) = \lim_{n \to \infty} 
  m(( e^{i \theta F \circ T^n}) g) = \avg( e^{i \theta F} ) = \varphi_X(\theta).
\]
Pointwise convergence of the characteristic \fn\ is equivalent to 
convergence in distribution.
\end{proof}

We illustrate this property by means of an example similar to the 
one presented after Definition \ref{def:GLM}. Set $F(x) := \{x\}$, where
$\{ \cdot \}$ denotes the fractional part of a real number. It is very easy 
to verify that $\avg( e^{i \theta F} ) = \int_0^1 e^{i \theta x} \, dx$, 
for all $\theta$. Therefore, Proposition \ref{prop-new-app} implies that
for a random initial condition $x$, w.r.t.\ any law $\nu \ll m$, the fractional 
part of $T^n(x)$ tends to be uniformly distributed in $[0,1)$, as $n \to \infty$.

Adding a small hypothesis on $F$ gives a result that is peculiar
to maps with neutral fixed points in whose neighborhoods lies
``most'' of the infinite invariant \me. This is the case for the Boole 
map too. 

\begin{definition} 
  We say that $F: \bR \to \bR$ is \emph{uniformly continuous at 
  infinity} if, for every $\varepsilon > 0$, there exist $\delta, Q > 0$
  such that
  \[
    \forall x,y \in \bR \setminus (-Q, Q) \mbox{ with  }|x-y| \le \delta,
    \quad |F(x) - F(y)| \le \varepsilon.
  \]
\end{definition} 

\begin{proposition} \label{prop-new-app2} 
  In addition to the hypotheses of Proposition \ref{prop-new-app},
  suppose that $F$ is bounded and uniformly continuous at infinity. 
  Also, for $k \in \bZ^+$, denote by
  \[
    \mathcal{A}_k F := \frac1{k} \sum_{j=0}^{k-1} F \circ T^j 
  \]
  the partial Birkhoff average of $F$. Then:
  \begin{itemize}
  \item[(i)] For $k$ fixed and $n \to \infty$, $\mathcal{A}_k F \circ T^n$ 
    converges in strong distributional sense to the random variable $X$ 
    defined in the statement of Proposition \ref{prop-new-app}.
  
  \item[(ii)] There exists a diverging subsequence $(k_n)$ such that,
    as $n \to\infty$, $\mathcal{A}_{k_n} F \circ T^n$ converges in strong 
    distributional sense to $X$.
  \end{itemize}
\end{proposition}

The significance of Proposition \ref{prop-new-app2} is that, for the Boole
map, the asymptotic distribution of an \ob\ is the same as that of any of
its partial time averages, at least for a large class of global \ob s. This 
phenomenon cannot occur for \m\ maps preserving a finite \me, cf.\ 
\cite[Section 3.2]{BGL17}. Proposition \ref{prop-new-app2} is proved 
almost exactly as Proposition 3.5 of \cite{BGL17}. The fact that the latter 
refers to maps $\bR^+ \to \bR^+$ and is formulated with a slightly stronger
assumption on $F$ plays no substantial role. 

Coming back to the example $F(x) = \{x\}$, one may observe that $F$
is not uniformly continuous at infinity. But we can define a very similar
\ob\ which is. For $x \in \bR$, let $\lfloor x \rfloor$ denote the maximum 
integer not exceeding $x$; for $x \ge 0$, $\lfloor x \rfloor$ is the integer
part of $x$, whence $x - \lfloor x \rfloor$ is its fractional part. Set
\[
  F_1(x) := \left\{
  \begin{array}{cl}
    x - \lfloor x \rfloor \,, & \mbox{ if $\lfloor x \rfloor$ is even}; \\[2pt]
    1 - x + \lfloor x \rfloor \,, & \mbox{ if $\lfloor x \rfloor$ is odd}.
  \end{array}
  \right.
\]
As in the previous example, $F_1 \circ T^n$ tends to be uniformly 
distributed on $[0,1]$, when $n \to \infty$. By Proposition \ref{prop-new-app2},
so do $\mathcal{A}_k F_1 \circ T^n$ and $\mathcal{A}_{k_n} F_1 \circ T^n$, 
for a certain subsequence $(k_n)$.

\begin{remark}
  A question that may be of interest is this: what random variables $X$ 
  can be strong distributional limits of $F \circ T^n$, for some $F$? The 
  answer is all random variables. In 
  fact, given a variable $X$ on a probability space $(\Omega, \bP)$, let 
  $\Phi (y) := \bP( X \le y )$ denote its distribution \fn\ and 
  \[
    \Phi^{-1}(x) := \inf \, \{ y \in \bR \,:\, \Phi (y) > x \}
  \]
  the right-continuous generalized inverse of $\Phi$. Here $x \in [0,1)$.
  It is well-known that, calling $U$ the uniform random variable on $[0,1]$,
  $\Phi^{-1} \circ U = X$ (in the sense of distributions). In other words,
  $\phi_X(\theta) = \int_0^1 e^{i \theta \Phi^{-1}(x)} \, dx$. Therefore, if
  $F(x) := \Phi^{-1}( x - \lfloor x \rfloor )$ is the periodization of
  $\Phi^{-1}$, then clearly
  \[
    \avg( e^{i \theta F} ) = \int_0^1 e^{i \theta \Phi^{-1}(x)} \, dx =
    \phi_X(\theta)
  \]
  and Proposition \ref{prop-new-app} applies.
  
  But $F$ may not be bounded and is surely not continuous --- for 
  a periodic \fn, continuity is the same as uniform continuity at infinity
  --- so, if the question is: for what variables $X$ do the assertions of 
  Proposition \ref{prop-new-app2} hold?, the answer is: at least all 
  variables which take values densely in an interval $[a,b]$ (this 
  means that $\bP( X \in [a,b] ) = 1$ and $\bP( X \in [c,d] ) > 0$, for all 
  $[c,d] \subseteq [a,b]$). In fact, for any such variable, $\Phi^{-1}$ 
  takes values in $[a,b]$ and is continuous. Therefore, the \fn\
  \[
    F_1(x) := \left\{
    \begin{array}{cl}
      \Phi^{-1}( x - \lfloor x \rfloor ) \,, & \mbox{ if $\lfloor x \rfloor$ 
      is even}; \\[2pt]
      \Phi^{-1}( 1 - x + \lfloor x \rfloor ) \,, & \mbox{ if $\lfloor x \rfloor$ 
      is odd}
    \end{array}
    \right.
  \]
  is bounded, 2-periodic, continuous and such that 
  $\avg( e^{i \theta F_1} ) = \phi_X(\theta)$. It thus satisfies the 
  hypotheses of Proposition \ref{prop-new-app2}.
\end{remark}

\appendix
\section{Appendix: Proof of Lemma \ref{th:lemma}}

We consider a generic map $S:\bR^+ \to \bR^+$, which is Markov and
piecewise $C^3$ with respect to the partition $I_{0} := [a,+\infty)$, 
$I_{1} := (0,a]$. It also satisfies the following hypotheses (H1)-(H4).

\begin{enumerate}
\item[(H1)] $S(I_{0}) = S(I_{1}) = \bR_0^+$;
\item[(H2)] $\forall x \in I_{0}$, $S'(x) > 1$; $S'(x)\to 1$ as $x\to \infty$; 
  $\forall x \in I_{1}$, $S'(x) < -1$;
\item[(H3)] $S$ preserves the Lebesgue measure $m$.
\end{enumerate}

Denoting $\phi_j := \left(S|_{I_j} \right)^{-1}$, for $j \in \{0,1\}$, we rewrite (H2) 
as follows

\begin{itemize}
\item[(H2)]
  \begin{itemize}
  \item[(i)] $\phi_{0}(0)=a$; $\forall x \in \bR_0^+$, $0<\phi'_{0}(x)<1$; 
  $\phi'_{0}(x)\to 1$ as $x\to \infty$;
  \item[(ii)] $\phi_{1}(0)=a$; $\forall x \in \bR_0^+$, $-1 < \phi'_{1}(x)<0$.
  \end{itemize}
\end{itemize}

Let $L : L^1(\bR^+, m) \to L^1(\bR^+, m)$ be the Perron-Frobenius operator
associated to $S$, i.e., the transfer operator relative to the Lebesgue \me\ $m$.
In formula:
\[
(L g)(x) = \phi_{0}'(x) \, g(\phi_{0}(x))  -\phi_{1}'(x) \, g(\phi_{1}(x)) .
\]
If we let $L$ act on $1$,  the \fn\ identically equal to 1 on
$\bR_0^+$, noting that strictly speaking $1$ does not belong to $L^1$, we
observe that (H3) is equivalent to $L 1 = 1$, or
\begin{enumerate}
\item[(H3)] $\forall x \in \bR_0^+$, $\phi_{0}'(x) - \phi_{1}'(x) =1$.
\end{enumerate}
Notice that (H3) implies that $\phi''_0(x) = \phi''_1(x)$ and $\phi'''_0(x) = \phi'''_1(x)$. 

We further assume that

\begin{itemize}
\item[(H4)] 
  \begin{itemize}
  \item[(i)] $\forall x\in \bR^+$, $1+2\phi'_{1}(x) > 0$;
  \item[(ii)] $\forall x\in \bR^+$, $\phi''_{1}(x) - (\phi'_{1}(x))^2 >0$;
  \item[(iii)] for all $x\in \bR^+$ one of the following is satisfied:
    \begin{itemize}
    \item[(a)] $\phi'''_1(x) + \phi''_1(x)>0$ and $3\, \phi''_{1}(x) - 
    (\phi'_{1}(x))^2 + \phi'_1(x)> 0$ 
    \item[or]
    \item[(b)] $\phi'''_1(x) + \phi''_1(x) - (\phi'_{1}(x))^2 >0$.
    \end{itemize}
  \end{itemize}
\end{itemize}

Observe that condition (H4)(i) is equivalent to $\phi_1'(x) > -1/2$ and 
thus supersedes one of the inequalities of (H2)(ii). We prefer to leave 
the latter as it is because the meaning of (H2) is that $S$ has one 
decreasing and one increasing branch, both expanding.

\begin{theorem} \label{thm:onedec}
Let $S$ be a piecewise $C^3$ map w.r.t.\ a partition $I_{0}=[a,+\infty)$, 
$I_{1}=(0,a]$ and assume that $S$ satisfies (H1)-(H4). Setting
\[
\mathcal{F} := \left\{ g:\bR^+\to \bR \,:\, g\in C^2 ,\, g>0 , \, g'<0 , \, g''+g'<0 
\right\} 
\]
we have $L(\mathcal{F}) \subset \mathcal{F}$.
\end{theorem}

\begin{proof}
The $C^2$-regularity and the positivity of $Lg$ for any $g\in \mathcal{F}$  
follow easily from the definition of the transfer operator. To verify the last two 
conditions we first write, using (H3),
\begin{equation} \label{der-prima}
\begin{split}
  (Lg)'(x) &=  \phi''_{0}(x) \, g(\phi_{0}(x)) - \phi''_{1}(x) \, g(\phi_{1}(x)) \\[6pt]
    &\qquad + (\phi'_{0}(x))^2 \, g'(\phi_{0}(x)) - (\phi'_{1}(x))^2 \, g'(\phi_{1}(x)) 
    \\[2pt]
  &= \phi''_{1}(x) \int_{\phi_1(x)}^{\phi_0(x)} \!\!\ g'(t)\, dt + g'(\phi_{0}(x))\, 
  (1+2\phi'_{1}(x)) \\
    &\qquad + (\phi'_{1}(x))^2 \, \Big[  g'(\phi_{0}(x)) - g'(\phi_{1}(x)) \Big] \\
  &= \phi''_{1}(x) \int_{\phi_1(x)}^{\phi_0(x)} \!\! g'(t)\, dt + (\phi'_{1}(x))^2 
  \int_{\phi_1(x)}^{\phi_0(x)} \! g''(t)\, dt \\[2pt]
    &\qquad + g'(\phi_{0}(x))\, (1+2\phi'_{1}(x))
\end{split}
\end{equation}
and
\begin{equation} \label{der-seconda}
\begin{split}
  (Lg)''(x) &= \phi'''_{1}(x) \int_{\phi_1(x)}^{\phi_0(x)} \!\! g'(t)\, dt + 
  3\, \phi''_1(x)\, \phi'_0(x)\, g'(\phi_{0}(x)) \\[2pt]
    &\qquad - 3\, \phi''_1(x)\, \phi'_1(x)\, g'(\phi_{1}(x))  
   + (\phi'_{0}(x))^3 g''(\phi_{0}(x)) \\[6pt]
   &\qquad -(\phi'_{1}(x))^3 \, g''(\phi_{1}(x)) .
\end{split}
\end{equation}

Now we prove that $(Lg)'(x) <0 $ for all $x\in \bR^+$. From \eqref{der-prima}, 
since $g''+g'<0$, we have
\[
(Lg)'(x) < \left[ \phi''_{1}(x) -  (\phi'_{1}(x))^2 \right] \int_{\phi_1(x)}^{\phi_0(x)}
\!\! g'(t)\, dt + g'(\phi_{0}(x))\, (1+2\phi'_{1}(x)).
\]
The claim follows from (H4)(i)-(ii) and the hypothesis $g'<0$. 

To show that $(Lg)''+(Lg)' <0 $ we begin from \eqref{der-seconda} and, 
using different arguments, we obtain the result for all $x\in \bR^+$ for 
which (H4)(iii) (a) or (b) are satisfied. In the case (a), we use $g''+g'<0$ 
to write
\begin{equation} \label{pippo}
\begin{split}
  &(Lg)''(x) + (Lg)'(x) = \left[ \phi'''_{1}(x) + \phi''_{1}(x) \right] 
  \int_{\phi_1(x)}^{\phi_0(x)} \!\! g'(t)\, dt \\
    &\qquad \qquad + (\phi'_{0}(x))^3 g''(\phi_{0}(x)) -(\phi'_{1}(x))^3 
    g''(\phi_{1}(x)) \\[6pt]
    &\qquad \qquad + \left[ 3\, \phi''_1(x)\, \phi'_0(x) + (\phi'_{0}(x))^2 \right] 
    g'(\phi_{0}(x)) \\[6pt]
    &\qquad \qquad - \left[ 3\, \phi''_1(x)\, \phi'_1(x) + (\phi'_{1}(x))^2 \right] 
    g'(\phi_{1}(x)) \\
  &\qquad < \left[ \phi'''_{1}(x) + \phi''_{1}(x) \right] \int_{\phi_1(x)}^{\phi_0(x)}
  \!\!\ g'(t)\, dt \\
    &\qquad \qquad + \left[ 3\, \phi''_1(x)\, \phi'_0(x) + (\phi'_{0}(x))^2 - 
    (\phi'_{0}(x))^3 \right] g'(\phi_{0}(x)) \\[6pt]
    &\qquad \qquad + \left[ -3\, \phi''_1(x)\, \phi'_1(x) - (\phi'_{1}(x))^2 + 
    (\phi'_{1}(x))^3 \right] g'(\phi_{1}(x)) .
\end{split}
\end{equation}
Let $x$ satisfy (H4)(iii)(a). We argue that all three terms in the 
last expression are negative. The first term is negative by the first condition 
of (H4)(iii)(a) and the fact that $g'<0$. The third term is negative by
the second condition of (H4)(iii)(a) and the inequalities $\phi'_1<0$ 
and $g'<0$. For the second term we use (H3) to get
\[
3\, \phi''_1(x)\, \phi'_0(x) + (\phi'_{0}(x))^2 - (\phi'_{0}(x))^3 = \phi'_0(x)
\left[ 3\, \phi''_1(x) - (\phi'_{1}(x))^2 - \phi'_1(x) \right] .
\]
Since $\phi'_0>0$ and $\phi'_1<0$, using the second condition
of (H4)(iii)(a), we obtain
\[
3\, \phi''_1(x)\, \phi'_0(x) + (\phi'_{0}(x))^2 - (\phi'_{0}(x))^3 > 
-2\, \phi'_0(x)\, \phi'_1(x) >0.
\]
Together with $g'<0$, this shows that the second term of
\eqref{pippo} is negative. In conclusion, $(Lg)''(x) + (Lg)'(x) <0$.

We now show that we obtain the same conclusion for all $x$ for 
which (H4)(iii)(b) is satisfied. Adding to \eqref{der-seconda} the
last expression of \eqref{der-prima}, we have
\begin{align*}
  &(Lg)''(x) + (Lg)'(x) = \left[ \phi'''_{1}(x) + \phi''_{1}(x) \right] 
  \int_{\phi_1(x)}^{\phi_0(x)} \!\! g'(t)\, dt \\
    &\qquad \qquad + (\phi'_{0}(x))^3 g''(\phi_{0}(x)) -(\phi'_{1}(x))^3 
    g''(\phi_{1}(x)) \\[8pt]
    &\qquad \qquad +3\, \phi''_1(x)\, \phi'_0(x)\, g'(\phi_{0}(x)) - 3\, 
    \phi''_1(x)\, \phi'_1(x)\, g'(\phi_{1}(x)) \\
    &\qquad \qquad +(\phi'_{1}(x))^2 \int_{\phi_1(x)}^{\phi_0(x)} \!\!\ 
    g''(t)\, dt + g'(\phi_{0}(x))\, (1+2\phi'_{1}(x)) .
\end{align*}
Using again the hypothesis $g''+g'<0$ we can write
\begin{align*}
  &(Lg)''(x) + (Lg)'(x) < \left[ \phi'''_{1}(x) + \phi''_{1}(x) - (\phi'_{1}(x))^2 \right] 
  \int_{\phi_1(x)}^{\phi_0(x)} \!\!\ g'(t)\, dt \\
    &\qquad \qquad + \left[ -3\, \phi''_1(x)\, \phi'_1(x) + (\phi'_{1}(x))^3 \right] 
    g'(\phi_{1}(x)) \\[8pt]
    &\qquad \qquad + \left[ 3\, \phi''_1(x)\, \phi'_0(x) + 1+2\phi'_1(x) - 
    (\phi'_{0}(x))^3 \right] g'(\phi_{0}(x)) .
\end{align*}
The first term on the above right hand side is negative by (H4)(iii)(b) 
and the inequality $g'<0$. The second term is negative because
 $\phi'_1<0$, $g'<0$ and, by (H4)(ii), $3\phi''_1(x) - (\phi'_{1}(x))^2
 > 2(\phi'_{1}(x))^2>0$. As for the third term we use once more that
 $g'<0$. Furthermore, via (H3),
\begin{align*}
  &3\, \phi''_1(x)\, \phi'_0(x) + 1+2\phi'_1(x) - (\phi'_{0}(x))^3 \\
  &\quad = \phi'_0(x) \left[ 3\, \phi''_1(x) -1  -2 \phi'_1(x) -(\phi'_{1}(x))^2 \right] + 
  1+2\phi'_1(x) \\
  &\quad = \phi'_0(x) \left[ 3\, \phi''_1(x) -(\phi'_{1}(x))^2 \right] + 
  (1+2\phi'_1(x))\, (1-\phi'_0(x)).
\end{align*}
The last expression above is positive because of the conditions
(H4)(i)-(ii) and the fact that $\phi'_0\in (0,1)$. This concludes the
proof that $(Lg)''(x) + (Lg)'(x) <0$ for all $x$ that satisfy (H4)(iii)(b). 
\end{proof}

In order to prove Lemma \ref{th:lemma} it remains to verify that
the map $\tilde T$ defined in \eqref{tilde-t} satisfies the assumptions of 
Theorem \ref{thm:onedec}.

First of all, $\tilde T$ is piecewise $C^3$ with respect to the partition 
$I_0=[1,+\infty)$ and $I_1=(0,1]$ and (H1) is evidently satisfied. 
Moreover we have
\[
\phi_{0}(x) = \frac{x+\sqrt{x^2+4}}{2} , \qquad \phi_{1}(x) = 
\frac{-x+\sqrt{x^2+4}}{2} ,
\] 
whence
\begin{align*}
\phi'_{0}(x) = \frac 12 + \frac{x}{2\, \sqrt{x^2+4}} , \quad & \quad \phi'_{1}(x) = 
-\frac 12 + \frac{x}{2\, \sqrt{x^2+4}} , \\
\phi''_{0}(x) = \phi''_{1}(x) &= \frac{2}{(x^2+4)^{\frac 32}} , \\
\phi'''_{0}(x) = \phi'''_{1}(x) &= -\frac{6\, x}{(x^2+4)^{\frac 52}} .
\end{align*}
From these expressions we can immediately verify (H2) and (H3). Moreover
\[
1+2\phi'_{1}(x) = \frac{x}{\sqrt{x^2+4}}  
\]
gives (H4)(i). Now let us compute
\[
\phi''_{1}(x) - (\phi'_{1}(x))^2 = -\frac{1}{4\, (x^2+4)^{\frac 32}} \left[ \sqrt{x^2+4}
\left( \sqrt{x^2+4} - x \right)^2 \!- 8 \right].
\]
One easily checks that
\begin{align*}
  & \frac{d}{dx} \left[ \sqrt{x^2+4} \left( \sqrt{x^2+4} - x \right)^2 \!- 8 \right] \\
  &\qquad = \left( 3x\, \sqrt{x^2+4} - 3x^2-8 \right) \frac{\sqrt{x^2+4}-x}
  {\sqrt{x^2+4}}  <0,
\end{align*}
for $x>0$. Hence, for all $x>0$,
\[
\sqrt{x^2+4} \left( \sqrt{x^2+4} - x \right)^2 \!- 8 < 0,
\]
which implies (H4)(ii).

In order to check assumption (H4)(iii), define:
\begin{align*}
  A_1 &:= \left\{ x\in \bR^+\, :\, \phi'''_1(x) + \phi''_1(x)>0\right\} ; \\
  A_2 &:= \left\{ x\in \bR^+\, :\, 3\, \phi''_{1}(x) - (\phi'_{1}(x))^2 + \phi'_1(x)> 0
    \right\} ; \\
  B &:= \left\{ x\in \bR^+\, :\, \phi'''_1(x) + \phi''_1(x) - (\phi'_{1}(x))^2 >0 
  \right\} .
\end{align*}
Clearly (H4)(iii) is satisfied if and only if
\[
\left( A_1\cap A_2 \right) \cup B = \bR^+.
\]
Also notice that it follows immediately from the definitions that 
$B \subseteq A_1$. Let us determine $A_1, A_2$ and $B$ in our case. 

First of all  $A_1 = \bR^+$ because
\[
\phi'''_1(x) + \phi''_1(x) = \frac{2\, (x^2-3x+4)}{(x^2+4)^{\frac 52}}.
\]

As for $A_2$,
\[
3\, \phi''_{1}(x) - (\phi'_{1}(x))^2 + \phi'_1(x) = \frac{ (x^2+4)^{\frac 32} -
\sqrt{x^2+4} \left( 2\, \sqrt{x^2+4}-x \right)^2 \!+ 24 }{ 4\, (x^2+4)^{\frac 32} }.
\]
Elementary calculus shows that the function in the numerator 
vanishes at $x=0$, is increasing up to $x = \sqrt{ \frac{7\sqrt{33}-9}{12} }$ 
and is decreasing afterwards. Evaluation at integers proves that the 
numerator is positive at least up to $x=5$, hence $(0,5) \subset A_2$. 
(Numerical computations give the approximate result $A_2=(0,x_1)$, 
with $x_1 \simeq 5.25166$.) Since $A_1 = \bR^+$ we conclude that
\begin{equation} \label{a-uffa}
(0,5) \subset A_1 \cap A_2.
\end{equation}

We now study $B$.
\begin{align*}
  & \phi'''_1(x) + \phi''_1(x) - (\phi'_{1}(x))^2 \\[4pt]
  &\qquad = \frac{ 8(x^2+4)-24x-(x^2+4)^{\frac 32} \left( \sqrt{x^2+4}-x \right)^2 } 
  { 4\, (x^2+4)^{\frac 52} } \\
  &\qquad = \frac{ (x^5+8x^3+4x^2+4x+16) - \sqrt{x^2+4}\, (x^4+6x^2+8) }
  { 2\, (x^2+4)^{\frac 52} }.
\end{align*}
It follows that
\[
\lim_{x\to +\infty} \left[ \phi'''_1(x) + \phi''_1(x) - (\phi'_{1}(x))^2 \right] = 
+\infty
\]
and that, for $x\in \bR^+$, the zeroes of $\phi'''_1(x) + \phi''_1(x) - 
(\phi'_{1}(x))^2$ are the same as the roots of the polynomial
\[
p(x) := 2x^6-7x^5+24x^4-56x^3+72x^2-76x+32.
\]
By applying synthetic substitution, we obtain
\[
p(x) = (x-4)\, (2x^5+x^4+28x^3+56 x^2+ 296x+1108) + 4464,
\]
which proves that the largest positive root of $p(x)$ is smaller than 4. 
This implies that
\begin{equation} \label{b-uffa}
(4,+\infty) \subset B.
\end{equation}
(Numerical computations show that $B = (0,x_2) \cup (x_3,+\infty)$, 
with $x_2 \simeq 0.690123$ and $x_3 \simeq 1.93158$.)

Inclusions \eqref{a-uffa} and \eqref{b-uffa} show that 
$\left( A_1\cap A_2 \right) \cup B = \bR^+$, which gives (H4)(iii)
and ends the proof of Lemma \ref{th:lemma}.
\qed

\end{document}